\documentclass{article}
\usepackage[utf8]{inputenc}
\usepackage{amssymb}
\usepackage{amsmath}
\usepackage{amsfonts}
\usepackage{graphicx}
\usepackage{xcolor}
\usepackage{marginnote}
\usepackage{titling}

  \makeatletter
\def\@fnsymbol#1{\ensuremath{\ifcase#1\or 1 \or \ddagger\or
   \mathsection\or \mathparagraph\or \|\or **\or \dagger\dagger
   \or \ddagger\ddagger \else\@ctrerr\fi}}
    \makeatother

\newcommand{\cp}{{\sf C}_{\sf p}}
\newcommand{\la}{{\langle}}
\newcommand{\ra}{{\rangle}}
\newcommand{\om}{{\omega}}
\newcommand{\fin}{[\omega]^{<\omega}}
\newcommand{\nnorm}{\left\lVert x\right\rVert}
\newcommand{\nnnorm}{\left\lVert \;\right\rVert}
\newcommand{\maxx}{\text{max}}
\newcommand{\minn}{\text{min}}
\newcommand{\infinitary}
{\mathcal{L}_{\omega_1,\omega}}
\newcommand{\sentl}{\text{Sent}(L)}
\newcommand{\val}{\text{Val}}
\newcommand{\modelm}{\mathfrak{M}}
\newcommand{\modeln}{\mathfrak{N}}
\newcommand{\logic}{\mathfrak{L}}
\newcommand{\integer}{\mathbb{Z}}
\newcommand{\rat}{\mathbb{Q}}
\newcommand{\real}{\mathbb{R}}
\newcommand{\s}{\mathcal{S}}
\newcommand{\oo}{\mathcal{O}}
\newcommand{\e}{\emptyset}
\newcommand{\strl}{\text{Str}(L)}
\newcommand{\strlprime}{\text{Str}(L')}
\newcommand{\modt}{\text{Mod}(T)}
\newcommand{\proof}{\text{\sf Proof.\;}}
\newcommand{\qed}{\hfill$\square$}
\newcommand{\U}{\mathcal{U}}
\newcommand{\V}{\mathcal{V}}
\newcommand{\C}{\mathcal{C}}
\newcommand{\B}{\mathcal{B}}

\newcommand{\ltpfim}{ltp_{\varphi,\mathfrak{M}}}
\newcommand{\rtpfin}{rtp_{\varphi,\mathfrak{N}}}
\newcommand{\slfi}{\mathcal{S}_{\varphi}^{l}}
\newcommand{\srfi}{\mathcal{S}_{\varphi}^{r}}
\newcommand{\lkom}{\mathcal{L}_{\kappa,\omega}}
\newcommand{\lomom}{\mathcal{L}_{\omega,\omega}}
\newcommand{\zfc}{{\sf ZFC}}

\newtheorem{definition}{\sf Definition}[section]
\newtheorem{remark}[definition]{\sf Remark}
\newtheorem{theorem}[definition]{\sf Theorem}
\newtheorem{lemma}[definition]{\sf Lemma}
\newtheorem{proposition}[definition]{\sf Proposition}
\newtheorem{corollary}[definition]{\sf Corollary}

\title{\sf Model Theory For $\cp$-Theorists }
\author{\sf Clovis Hamel\footnotemark[1] \ and Franklin D. Tall\thanks{
Research supported by NSERC grant A-7354.
}}
\date{\today}

\begin{document}

\maketitle

{\it This paper is dedicated to Prof. A. V. Arhangel’ski\u\i, whose research has inspired the second author for more than fifty years, and who is the founder of $\cp$-theory as a coherent subfield of general topology.}

\begin{abstract}
We survey discrete and continuous model-theoretic notions which have 
important connections to general 
topology. We present a self-contained exposition of several interactions between continuous logic and $\cp$-theory   which have applications to a classification problem involving Banach spaces not including $c_0$ or $l_p$,  following recent results obtained by P. Casazza and J. Iovino for compact continuous logics.
Using $\cp$-theoretic results involving Grothendieck spaces and double limit conditions, we extend their results to a broader family of logics, namely those with a first countable weakly Grothendieck space of types.
We pose $\cp$-theoretic problems which have model-theoretic implications.
\end{abstract}

 \section{\sf Introduction}
 It is perhaps impossible to find an area of modern mathematical research in which topology does not play a relevant role. This, of course, does not imply that a particular instance of this fact has to be of interest for topologists in general. This survey paper is intended to serve as an introduction to model theory, especially to continuous logics, for those who have had little or no previous contact with this branch of mathematical logic, showing how several model-theoretic results can be equivalently translated into topological statements and how answering $\cp$-theoretic questions can lead to solving problems in the context of continuous logics. In particular, we present new results concerning Gowers’ problem \cite{gowers} about the definability of a pathological Banach space.
 
Model theory is mainly concerned with theories and models. Theories are sets of axioms and their logical consequences, and a model of a theory is a mathematical structure satisfying the axioms. A model theorist studies the connections between syntax and semantics: the syntax refers to certain allowable strings of symbols whereas the semantics are about the interpretations of these symbols in a structure. Classical model theory is centred around first-order logic, the logic mathematicians usually work with, in which the Compactness Theorem holds: if every finite set of formulas of a theory has a model, then the theory has a model. Given any logic and a language $L$, two topological spaces are of interest: the space of $L$-structures and the space of types. The Compactness Theorem holds if and only if those spaces are compact.

Let $c_{00}$ denote the space of eventually zero sequences of real numbers, $c_0$ the space of sequences for which the limit is zero, and $l_p$  the space of sequences ($\la{x_n \colon n<\om}\ra$) of real numbers such that $\sum_{n<\om}|x_n|^p<\infty$.
The main topic presented here is a classification problem in continuous logic, which involves a fair amount of $\cp$-theory. In 1974, B. Tsirelson \cite{tsirelson} constructed a Banach space which does not include a copy of $l_p$ or $c_0$. The construction involves a process of approximation that had its inspiration in Cohen’s forcing method. 
What is now called \textit{Tsirelson’s space} is due to T. Figiel and W. Johnson \cite{figieljohnson} and is the dual of the original space constructed by Tsirelson: let $\{x_n \colon n<\om\}$ be the canonical basis for $c_{00}$. 
If $x=\sum_{n<\om}a_nx_n\in{c_{00}}$ and $E, F\in{\fin}$, denote $\sum_{n\in{E}}a_nx_n$ by $Ex$ and write $E\leq{F}$ if and only if $\text{max} E\leq{\minn F}$. The norm $\left\lVert \ \right\rVert_T$ of Tsirelson's space $T$ is constructed by an approximation process and is the unique norm satisfying: $\nnorm_T=
\maxx\{\nnorm_{c_{00}},\frac{1}{2}
\maxx\{\sum_{i<k}\left\lVert {E_ix} \right\rVert_T \colon \{k\}\leq{E_1}<\dots<E_k\}\}$ (see \cite{figieljohnson} for the details). This equation captures the essence of what analysts call \emph{implicit definability}: the norm of Tsirelson’s space $\left\lVert\;\right\rVert_T$ appears on both sides of the previous defining equality. Intuitively, to obtain an implicitly definable Banach space, one has to write down a “definition”, which may very well mention the object to be defined on both sides of the defining equality, and then prove existence and uniqueness of the “defined” object, as opposed to explicitly definable objects which are simply defined by a formula. For a detailed exposition of Tsirelson’s space, the reader is referred to P. Casazza and T. Shura \cite{casazzashura}. 
After Tsirelson, many other pathological Banach spaces were constructed using similar techniques. The motivational problem for our work was posed by E. Odell and popularized by W. T. Gowers \cite{gowers}: is it true that if a Banach space is explicitly definable, then it includes a copy of $l_p$ or $c_0$?
Of course, this is a vague question that requires a rigorous formulation. The first problem one faces is to choose the logic to work with; even for first-order logic the answer was unknown until \cite{casazzaiovino}. \emph{Continuous logics} are the natural candidates for logics as they provide the necessary tools for the analyst’s $\varepsilon$-play and approximations. P. Casazza and J. Iovino \cite{casazzaiovino} presented both a formulation and a solution of this problem in the context of compact continuous logics, i.e. continuous logics which are finitary in nature.

The purpose of addressing parts of the work in \cite{casazzaiovino} here is twofold: to isolate the most relevant interactions between topology and model theory in their paper, and to provide the reader with definitions and approaches that may be more suitable for extending the contents of \cite{casazzaiovino} to non-compact continuous logics such as continuous $\infinitary$ (defined below). A $\cp$-theoretic problem
involved in this classification problem concerns establishing conditions for $X$
such that $\bar{A}$ being compact, for a subspace $A$ of $C_p(X)$,
is  equivalent to a double limit condition holding. The literature on this problem includes A. Grothendieck \cite{grothendieck}, V. Pt\'ak \cite{ptak}, N. Young \cite{young}, A. V. Arhangel’ski\u\i \  \cite{arhangelskii} and V. V. Tkachuk \cite{tkachuk}.

Sections 1 and 2 provide a self-contained introduction to classical results in \emph{discrete} model theory (discrete refers to classical logic where truth values are either 0 or 1) that will then be generalized in the context of continuous logics for metric structures following \cite{benyaacoviovino}, \cite{benyaacovusvyatsov} and \cite{eagle}. Sections 3 and 4 provide a background beyond two-valued first-order logic for the subsequent results. Section 5 includes an exposition of the interactions between $\cp$-theory and continuous logic in \cite{casazzaiovino} together with a couple of $\cp$-theoretic questions and problems that would have implications in continuous model theory.
$\cp$-theorists who find Sections 3 and 4 heavy-going may want to temporarily skip ahead to the discussion of Grothendieck spaces 
in Section 5. Section 6 includes generalizations and some remarks regarding the same classification problem that is addressed in \cite{casazzaiovino}, in the context of infinitary continuous logic.
 
 \section{\sf Preliminaries - What is Model Theory?}
 We first set out some elementary notions that will be used throughout. For a more detailed exposition, see \cite{tentziegler}. A \emph{language} is a set of constants, function symbols and relation symbols, e.g. $\{\in\}$ is the language of set theory, where $\in$ is a binary relation, and $\{\bar{e}, *, ^{-1}\}$ is the language of group theory, where $\bar{e}$ is a constant, $*$ is a binary relation and $^{-1}$ is an unary function.
 \emph{First-order logic} is the usual logic used in mathematics: given a language $L$ 
 formulas are built recursively as certain finite strings of symbols built using the members of $L$, parentheses, variables, and the logical connectives
 $\wedge, \vee, \to, \leftrightarrow, \neg$ and quantifiers $\exists$, $\forall$.
 For the rest of this section let $L$ be a fixed language.  Sometimes it will be useful to denote a formula $\varphi$ by $\varphi^1$ and $\neg{\varphi}$ by $\varphi^0$.

An $L$-structure $\modelm$ is a set $M$, called the \emph{universe of} $\modelm$ together with an \emph{interpretation} of elements of $L$, i.e. a function which assigns an element of $M$ to each constant symbol, a function from $M^n$ to $M$ to each $n$-ary function symbol, and a subset of $M^n$ to each $n$-ary relation symbol. Given an $L$-formula $\varphi$, we denote by $\varphi^{\modelm}$ the interpretation of $\varphi$ in $\modelm$.
A \emph{substructure} $\modeln$ of $\modelm$ is just an $L$-structure, the universe of which is $N$, a subset of $M$, containing all the same interpretations of symbols from $\modelm$. For instance, in the language of group theory, the formula $(\forall x)(\exists y)(x*y=\bar{e})$ is interpreted as $(\forall x\in{\integer})(\exists y\in{\integer})(x+y=0)$ in the additive group $\integer$.
 
A variable is said to be \emph{free} in a formula if it is not under the scope of any quantifier; a \emph{sentence} is a formula without free variables. A \emph{theory} is a set of sentences and their logical consequences; a theory is \emph{consistent} if it has a model and a maximal consistent theory is said to be \emph{complete} (Note consistency is often defined proof-theoretically: not every sentence can be proved. The Completeness Theorem for first-order logic asserts that consistency is equivalent to satisfiability). For example, the axioms and theorems of group theory constitute a theory. If $T$ is a theory, the \emph{satisfaction} relation $\modelm\models{T}$ is defined recursively to mean $\modelm$ is a model of $T$, i.e. $T$ is true in $\modelm$. As examples, $G\models\text{ axioms of group theory}$, for any group $G$, and 
$\real\models(\forall x)(\forall y)(x*y=y*x)$.
A submodel $\modeln$ of $\modelm$ is \emph{elementary} if for any $L$-formula 
$\varphi(x_1,\dots,x_n)$ and $a_1,\dots,a_n\in{\modeln}$, we have 
$\modeln\models\varphi(a_1,\dots,a_n)$ if and only if 
$\modelm\models\varphi(a_1,\dots,a_n)$; in this case we write $\modeln\preceq{\modelm}$.
Two models $\modelm$ and $\modeln$ are \emph{elementarily equivalent} if for any $L$-sentence $\varphi$ we have $\modelm\models{\varphi}$ if and only if 
$\modeln\models{\varphi}$; this is denoted
by $\modelm\equiv{\modeln}$.

If $\modelm$ is an $L$ structure and $A\subseteq{M}$, we can extend the language $L$ to $L(A)$ by adding a constant symbol 
$\dot{a}$ for each $a\in{A}$. In this case, an $L(A)$-formula is simply an $L$-formula with parameters from $A$.

Given a language $L$, two interesting topological spaces arise: the space of types and the space of $L$-structures. We shall now introduce them and see how the Compactness Theorem for first-order logic is equivalent to each of those spaces being compact.
\begin{definition}
Let $\modelm$ be an $L$-structure and $A\subseteq{M}$. A complete $n$-type over $A$ in $\modelm$ is a maximal satisfiable set of $L(A)$-formulas in the variables $x_1, \dots, x_n$;
i.e. a set $p(x_1, \dots, x_n)$ of formulas
$\varphi(x_1,\dots,x_n)$ such that there are
$a_1,\dots,a_n\in{M}$ such that for every $\varphi\in{p}$ we have $\modelm\models{\varphi(a_1,\dots,a_n)}$. The set of all $n$-types is denoted by  $\s^n(A)$.
\end{definition}
Give a formula $\varphi=\varphi(x_1,\dots,x_n)$, we define
$$[\varphi]=\{p\in{\s^n(A)} \colon \varphi\in{p}\}$$ where the superscript is omitted if it is $1$.

Note that all sets of the form $[\varphi]$ 
form a basis for a topology on $\s^n$. Moreover, since the negation: $\neg$ is available in first-order logic, each $[\varphi]$ is clopen. We will soon be dealing with other logics in which negation is not at our disposal and the $[\varphi]$'s
will only constitute a basis for the closed sets. Most of the time, it is enough to study $1$-types, and then the same proofs can be carried out for $n$-types.
\begin{definition}
Let $\strl$ be the set of all equivalence classes under $\equiv$ of $L$-structures.
For each theory $T$ let 
$[T]=\{\bar{\modelm}\in{\strl} \colon \modelm\models{T}\}$ where $\bar{\modelm}$ is the equivalence class of $\modelm$. All the sets of the form $[T]$ constitute a basis
for the closed sets of the topology on $\strl$ known as the space of $L$
-structures. We write $[\varphi]$ instead of $[\{\varphi\}]$.
\end{definition}
\begin{remark}\label{firstremark}
{\rm Considering equivalence classes under $\equiv$ instead of models is a way to go around foundational issues: the upward 
L\"owenheim-Skolem theorem states that if a theory has a model of cardinality $\kappa$ then it has a model of cardinality $\lambda$ for every $\lambda\geq{\kappa}$.
Consequently, dropping the modulo $\equiv$
would make $\strl$ and $\modt$ proper classes.
Considering equivalence classes is equivalent to simply considering the set of all consistent theories, which is a set of cardinality $\leq{2^{|L|}}$.
A topological consequence of mod-ing out by $\equiv$ is that we guarantee that the topology on $\strl$ is Hausdorff. This is not necessarily the case otherwise, since there could be two different elementarily equivalent structures $\modelm$ and $\modeln$ which are not topologically distinguishable.}
\end{remark}
Now we state a result which is a cornerstone of first-order logic:
\begin{theorem}[The Compactness Theorem]
Let $T$ be an $L$-theory in first-order logic. If $T$ is finitely satisfiable, i.e. for any $\Delta\in{[T]^{<\omega}}$ there is an $L$-structure $\modelm_{\Delta}$ satisfying $\modelm_{\Delta}\models{T}$, then $T$ is consistent.
\end{theorem}
\begin{proof}
Notice that if equivalent $L$-formulas are identified, we get a Boolean algebra and then its corresponding Stone space $\mathbb{S}$
is compact. An element of the Stone space is simply a complete $L$-theory.
We prove the contrapositive. Suppose $T$ 
is inconsistent and take an enumeration
$T=\{\varphi_{\alpha} \colon \alpha<\kappa\}$.
Since $T$ is finitely satisfiable, 
$\oo=\{[\neg{\phi_{\alpha}}] \colon \alpha<\kappa\}$, where 
$[\neg{\varphi_{\alpha}}]=\{x\in{\mathbb{S}} \colon \neg{\varphi_{\alpha}}\in{x}\}$ covers $\mathbb{S}$ since $\bigcap\{[\varphi_{\alpha}] \colon \alpha<\kappa\}=\e$.
If $\{[\neg{\varphi_{\alpha_i}}] \colon i<n\}$ is a finite subcover of $\oo$, then 
$\bigcap\{[\neg{\varphi_{\alpha_i}}] \colon i<n\}=\e$ and thus $\{\varphi_{\alpha_i} \colon i<n\}$ is a finite inconsistent subset
of $T$.
\end{proof}
\qed

The relationship between the Compactness Theorem and the space of structures and the space of  $n$-types is fundamental. Basically, a logic satisfies the Compactness Theorem if and only if its space of structures is compact, and if and only if its space of $n$-types is compact.
To formulate this precisely, we would need to formulate precisely what a logic is, as is done for example in \cite{eagletall}. We won't do this here; the following two theorems will easily be seen to apply to whatever logics are considered in this paper.
\begin{theorem}
Given a logic and a language $L$, every finitely satisfiable set of $L$-formulas
is satisfiable if and only if the space of
$L$-structures $\strl$ is compact.
\end{theorem}
\proof 
$\Rightarrow)$
Let $\{[\varphi_{\alpha}] \colon \alpha<\kappa\}$ be a centred family (i.e. every finite subfamily has non-empty intersection) of basic closed sets of $\strl$.
Since $[\varphi_{\alpha}]\cap{[\varphi_{\beta}]}=[\varphi_{\alpha}\wedge\varphi_{\beta}]$, $T=\{\varphi_{\alpha}\colon \alpha<\kappa\}$ is finitely satisfiable. The Compactness Theorem yields a model $\modelm$
of $T$, i.e. $\modelm\in{\bigcap\{[\varphi_{\alpha}] \colon \alpha<\kappa\}}$.
$ \Leftarrow)$
If $T=\{\varphi_{\alpha}\colon \alpha<\kappa\}$ is finitely satisfiable, then 
$\{[\varphi_{\alpha}] \colon \alpha<\kappa\}$ is a centred family and any 
$\modelm\in{\bigcap\{[\varphi_{\alpha}] \colon \alpha<\kappa\}}$ is a model of $T$.
\qed
\begin{theorem}
Given a logic, a language $L$ and a $n<\om$,
every finitely satisfiable set of $L$-formulas
is satisfiable if and only if the space of types $\s^n$ is compact.
\end{theorem}
\proof
$\Rightarrow)$ 
Let $\{[\varphi_{\alpha}] \colon \alpha<\kappa\}$ be a centred family of basic closed sets of $\s^n$.
Then $\{\varphi_{\alpha}\colon \alpha<\kappa\}$ is finitely satisfiable, and so it is satisfiable, i.e. there is an $L$-structure $\modelm$ with $a_1,\dots,a_n\in{M}$ such that for every
$\alpha<\kappa$ we have 
$\modelm\models\varphi_{\alpha}(a_1,\dots,a_n)$. By an application of Zorn's lemma, $T$ can be
extended to an $n$-type $p$.
Clearly, $p\in{\bigcap\{[\varphi_{\alpha}] \colon \alpha<\kappa\}}$.
$\Leftarrow)$
If $T=\{\varphi_{\alpha}\colon \alpha<\kappa\}$ is finitely satisfiable
then $\{[\varphi_{\alpha}] \colon \alpha<\kappa\}$ is a centred family of closed sets in $\s^n$.
Take $p\in{\bigcap\{[\varphi_{\alpha}] \colon \alpha<\kappa\}}$; then $p$ is a type and so it is satisfied by a model $\modelm$.
In particular $\modelm\models\varphi_{\alpha}$
for each $\alpha<\kappa$.
\qed

\section{\sf Stability and Definability in First-Order Logic}
It is a fruitful area of research to consider logics other than first-order logic. However, one cannot expect the Compactness Theorem to hold in every case. The following sections are devoted to presenting an exposition of logics with stronger expressive powers and showing how to use different topological tools when the Compactness Theorem does not hold.

We first examine the concept of \emph{stability} which has been a keystone of model theory since Shelah \cite{shelah} introduced it. The purpose of this section is to show that studying stability is a natural way to approach definability.
\begin{definition}
If $\boldsymbol{x}$ is a $n$-tuple of variables, we write $l(\boldsymbol{x})=n$. Let $\modelm$ be an $L$-structure, $\varphi(\boldsymbol{x},\boldsymbol{y})$
an $L$-formula where $l(\boldsymbol{x})=n$ and 
$l(\boldsymbol{y})=m$, $A\subseteq{M}$ and $\boldsymbol{b}\in{M^n}$. An $n$-$\varphi$-type over $A$ in $\modelm$ is a set of formulas
$tp_{\varphi}(\boldsymbol{b},A,\modelm)=
\{\varphi^{t}(\boldsymbol{x},\boldsymbol{a})\colon \boldsymbol{a}\in{A^m} \wedge t<2 \wedge \modelm\models \varphi^t(\boldsymbol{b},\boldsymbol{a})\}$.
Thus, the space of $n$-$\varphi$-types is
$\s_{\varphi}^n=\{tp_{\varphi}(\boldsymbol{b},A,\modelm) \colon \boldsymbol{b}\in{M^n}\}$.
Again, we omit the $n$ from the notation when it is $1$.
\end{definition}
\begin{definition}
Let $T$ be an $L$-theory
\begin{itemize}
    \item [\sf (i)] A model $\modelm$ of $T$ is stable on an infinite cardinal $\lambda$ if for every $A\subseteq{M}$ such that $|A|\leq{\lambda}$ we have $|\s(A)|\leq{\lambda}$.
    \item [\sf (ii)] $T$ is stable if all its models are stable.
    \item [\sf (iii)] An $L$-formula $\varphi\in{T}$ is stable in $\modelm$ on an infinite cardinal $\lambda$ if for every $A\subseteq{M}$ such that $|A|\leq{\lambda}$ we have that $|\s_{\varphi}(A)|\leq{\lambda}$.
\end{itemize}
\end{definition}
Since the previous definition of stability is usually cumbersome to work with, we present several equivalent definitions. For a more detailed exposition, see \cite{shelah2}. Despite a superficial resemblance to the concept of stability in $\cp$-theory, we have not been able to find a nice connection between the two.
\begin{definition}
An $L$-formula $\varphi(\boldsymbol{x},\boldsymbol{y})$ has the order property in $\modelm$ if there are sequences $\la{\boldsymbol{a}_i \colon i<\om}\ra$ 
and $\la{\boldsymbol{b}_j \colon j<\om}\ra$ of tuples from $M$ such that, for every $i,j<\om$, 
$\modelm\models\varphi(\boldsymbol{a}_i,\boldsymbol{b}_j)$ if and only if $i<j$.
\end{definition}

\begin{remark} \label{remarkramsey}
{\rm It is a consequence of the Compactness Theorem and Ramsey’s Theorem that having the order property is symmetric in $\boldsymbol{x}$ and $\boldsymbol{y}$, i.e.~if $\varphi(\boldsymbol{x},\boldsymbol{y})$ has the order property, so does $\varphi(\boldsymbol{y},\boldsymbol{x})$.
For a detailed proof of this fact, see \cite{shelah2} or \cite{tentziegler}.}
\end{remark}
\begin{theorem}
Let $T$ be an $L$-theory. Then the following are equivalent: 
\begin{itemize}
    \item [\sf (i)] $T$ is stable on some $\lambda\geq{\aleph_0}$.
    \item [\sf (ii)] $T$ is stable on every $\lambda\geq{\aleph_0}$.
    \item [\sf (iii)] Every formula $\varphi(\boldsymbol{x},\boldsymbol{y})$ in $T$ is stable on some $\lambda\geq{\aleph_0}$ in every model of $T$.
    \item [\sf (iv)] Every formula $\varphi(\boldsymbol{x},\boldsymbol{y})$ in $T$ is stable on every $\lambda\geq{\aleph_0}$ in every model of $T$.
    \item [\sf (v)] For every formula $\varphi(\boldsymbol{x},\boldsymbol{y})$ in $T$ and every model $\modelm$ of $T$, $\varphi$ does not have the order property in $\modelm$.
\end{itemize}
\end{theorem}
For a proof see Shelah \cite{shelah2}.

We now state the classical theorem from first-order model theory which relates the concepts of stability and definability.
\begin{definition}
Let $\modelm$ be an $L$-structure and $A,B\subseteq{M}$
\begin{itemize}
    \item [\sf (i)] We say that $A$ is $B$-definable if there is an $L(B)$-formula $\varphi(x)$ such that $A=\{x\in{M} \colon \modelm\models\varphi(x)\}$.
    \item [\sf (ii)] An $n$-$\varphi$-type over $A$ in $\modelm$ is $B$-definable if there is an $L$-formula $\psi(\boldsymbol{x},\boldsymbol{y})$ and a tuple $\boldsymbol{b}$ from $B$ so that for every $\boldsymbol{a}\in{A^n}$, 
    $\varphi(\boldsymbol{x},\boldsymbol{a})\in{p}
    $ if and only if $\modelm\models\psi(\boldsymbol{a},\boldsymbol{b})$.
\end{itemize}
\end{definition}
The following is a classical theorem, the proof of which can be found in the literature. See for instance, \cite{shelah2} or \cite{tentziegler}.
\begin{theorem}[The Definability Theorem]
An $L$-formula $\varphi$ is stable in $\modelm$ if and only if for every $A\subseteq{M}$ and for every $p\in{\s_{\varphi}(A)}$, $p$ is $A$-definable. 
\end{theorem}
As done in \cite{casazzaiovino} for compact continuous logics, a reasonable scheme for approaching the problem of the definability of Tsirelson-like spaces consists of the following:
\begin{itemize}
    \item [\sf (i)] Extend the notions of stability and definability from first-order logic to the logic in question.
    \item [\sf (ii)]Prove an analogue of the Definability Theorem.
    \item [\sf (iii)] Establish a relationship between stability and double limit conditions. \emph{This is where $\cp$-theory plays a fundamental role.}
    \item [\sf (iv)] Decide on the definability in that logic of Tsirelson’s space (or any other Banach space) by using the double limit condition.
\end{itemize}
Given a model $\modelm$, it is convenient to also view an $L$-formula $\varphi(x,y)$ as a function $\varphi \colon M\times{M}\to 2$ such that $\varphi(a,b)=1$ if and only if $\modelm\models\varphi(a,b)$. Then a $\varphi$-type $q=tp_{\varphi}(a,A,\modelm)$ can be identified with the function 
$a\rightarrow\varphi(a,b)$ (the $\varphi$-type is recovered by taking the pre-image of $\{1\}$).

In the following chapters, stability will be defined in terms of double-(ultra)limit conditions. We remind the reader that if $\U$ is an ultrafilter over $\kappa$ and $\langle x_\alpha : \alpha<\kappa \rangle$ is a $\kappa$-sequence, we write $\lim_{\alpha\to\U}x_\alpha =x$ if for every neighbourhood $U$ of $x$, $\{\alpha : x_\alpha\in U\}\in \U$. As a motivation, we see that this is the case in first-order logic, following Iovino \cite{iovino}:
\begin{lemma}\label{3.8}
Let $\varphi\colon A\times{B}\to [0,1]$ be a function, $\la{a_n \colon n<\om}\ra$ and 
$\la{b_m \colon m<\om}\ra$ be sequences in $A$ and $B$ respectively and $\U$ and $\V$ ultrafilters on $\om$.
If
$$\lim_{n\to\U}\lim_{m\to\V}\varphi(a_n,b_m)=\alpha \text{ and }
\lim_{m\to\V}\lim_{n\to\U}\varphi(a_n,b_m)=\beta
$$
then there are subsequences 
$\la{a_{n_i} \colon i<\om}\ra$ and
$\la{b_{m_j} \colon j<\om}\ra$ such that
$$\lim_{\stackrel{i<j}{i\to\infty}}\varphi(a_{n_i},b_{m_j})=\alpha \text{ and }
\lim_{\stackrel{j<i}{j\to\infty}}\varphi(a_{n_i},b_{m_j})=\beta.
$$
\end{lemma}
\begin{proof}
Given $\varepsilon>0$ we have
\begin{itemize}
    \item [\sf (i)] $(\exists U_{\varepsilon}\in{U})(\forall n\in{U_{\varepsilon}})|
    \lim_{m\to\V}\varphi(a_n,b_m)-\alpha|<
    \varepsilon$
    \item [\sf (ii)] $(\exists V_{\varepsilon}\in{V})(\forall m\in{V_{\varepsilon}})|
    \lim_{n\to\U}\varphi(a_n,b_m)-\beta|<
    \varepsilon$
    \item [\sf (iii)] $(\forall m<\om)(\exists \U^m_\varepsilon\in\U)(\forall n\in\U^m_\varepsilon)|
    \varphi(a_n,b_{m})-\lim_{n\to \U}\varphi(a_n,b_{m})|<\varepsilon$.
    \item [ \sf(iv)] $(\forall n<\om)(\exists \V^n_\varepsilon\in\V)(\forall m\in\V^n_\varepsilon)|
    \varphi(a_n,b_m)-\lim_{m\to \V}\varphi(a_n,b_m)|<\varepsilon$.
\end{itemize}
Notice that {\sf (i)} and {\sf (ii)} follow from the convergence hypothesis, and 
{\sf (iii)} and {\sf (iv)} follow from the compactness of [0,1].
Construct recursively two increasing sequences
$\la{n_i \colon i<\om}\ra$ and $\la{m_j \colon j<\om}\ra$ such that $n_0\in{U_1}$, $m_0\in{V_1\cap{V_1^{n_0}}}$ and 
$n_i\in{U_{1/i+1}}
\cap{\bigcap_{k<i}U_{1/i+1}^{m_k}}$ and 
$m_j\in{U_{1/j+1}}
\cap{\bigcap_{k\leq{j}}V_{1/j+1}^{n_k}}$.
Now consider the subsequences 
$\la{a_{n_i} \colon i<\om}\ra$ and 
$\la{b_{m_j} \colon j<\om}\ra$.
The construction shows that
$\lim_{\stackrel{i<j}{i\to\infty}}
\varphi(a_{n_i},b_{m_j})=\alpha$ and similarly $\lim_{\stackrel{j<i}{j\to\infty}}\varphi(a_{n_i},b_{m_j})=\beta$.
\end{proof}
\qed

\noindent
The previous result is used in \cite{iovino} to prove:

\begin{theorem}
Given a formula $\varphi(x,y)$, the following are equivalent:
\begin{itemize}
    \item [\sf (i)] $\varphi(x,y)$ is stable.
    \item [\sf (ii)] For any sequences $\la{a_n \colon n<\om}\ra$ and $\la{b_m \colon m<\om}\ra$
    in $M$, and any ultrafilters $\U$ and $\V$ on $\om$, $\lim_{n\to \U}\lim_{m\to \V}\varphi(a_n,b_m)=\lim_{m\to \V}\lim_{n\to \U}\varphi(a_n,b_m)$.
\end{itemize}
\end{theorem}
\begin{proof}
${\sf (ii)}\rightarrow{\sf (i)}$ Assume that $\varphi$ has the order property witnessed by $\la{a_n \colon n<\om}\ra$ and $\la{b_m \colon m<\om}\ra$. Then
$\lim_{m\to\infty}\varphi(a_n,b_m)=\lim_{m\to\V}\varphi(a_n,b_m)=1$ for any fixed $n<\om$.
Similarly, $\lim_{n\to\infty}\varphi(a_n,b_m)=\lim_{n\to\U}\varphi(a_n,b_m)=0$ for any fixed $m<\om$.
So $\lim_{n\to \U}\lim_{m\to \V}\varphi(a_n,b_m)=1$ and 
$\lim_{m\to \V}\lim_{n\to \U}\varphi(a_n,b_m)=0$.
${\sf (i)}\rightarrow{\sf (ii)}$
By the lemma, there are subsequences
$\la{a_{n_i} \colon i<\om\ra}$ and
$\la{b_{m_j} \colon j<\om}\ra$ such that:
$$\lim_{\stackrel{i<j}{i\to\infty}}
\varphi(a_{n_i},b_{m_j})=1\text{ and }
\lim_{\stackrel{j<i}{j\to\infty}}
\varphi(a_{n_i},b_{m_j})=0
$$
Equivalently, there are $N_0, N_1<\om$ such that if $j>i>N_0$, then $\varphi(a_{n_i},b_{m_j})=1$ and 
if $i>j>N_1$, then 
$\varphi(a_{n_i},b_{m_j})=0$.
If $N$ is bigger than both $N_0$ and $N_1$, 
then $\la{a_{n_i} \colon i<N}\ra$ and 
$\la{b_{m_j} \colon j<N}\ra$ witness that
$\varphi$ has the order property.
If the limits have the other possible values, use Remark \ref{remarkramsey}.
\end{proof}
\qed

\section{\sf Continuous Logics for Metric Structures} 

The logics that we will consider have their motivation in overcoming some notorious limitations of first-order
logic, namely that the only possible truth values lie in $\{0,1\}$, leaving little room for the usual $\varepsilon$-play and approximations which are fundamental in analysis, and that only finite conjunctions are allowed,
restricting certain recursive definitions. We will first look at continuous logics, particularly logics for metric
structures, where the first limitation is overcome by allowing each real in $[0,1]$ as a possible truth value. We will deal with the second difficulty later when we introduce the continuous version of $\infinitary$, in which
countable conjunctions are allowed. The origins of continuous logic lie in W. Henson’s approximate
satisfaction \cite{henson}. A $[0,1]$-valued logic called \emph{continuous first-order logic} has been an active area of
research for the last decade since it was introduced by I. Ben Yaacov and A. Usvyatsov \cite{benyaacovusvyatsov} following ideas
of C. C. Chang and H. J. Keisler \cite{changkeisler}, and Henson and Iovino \cite{hensoniovino}. We will deal with variations and
generalizations introduced by Ben Yaacov and Iovino \cite{benyaacoviovino} and C. J. Eagle \cite{eagle}, following mostly the last one.

For the topologist reader who is not so interested in analysis, we emphasize that we could just consider
the usual first-order logic (and later, $\infinitary$). We then do not have to worry about continuous connectives,
non-trivial metric structures, unexpected definitions of definability, etc. We will still use non-trivial theorems of $\cp$-theory concerning ${\sf C}_{\sf p}(X,2)$, and prove non-trivial new results about definability in the
$2$-valued infinitary logic $\infinitary$.

In the discrete case, the interpretation in a model $\modelm$ of an $n$-ary relation, also called an $n$-ary predicate, was defined as a subset of $M^n$ and so it can also be regarded as a function $M^n\to 2$ in the natural way.
We see that this last option is the most natural one in the continuous case.

\begin{definition}
Assume for the sake of notational simplicity that all metric
structures have diameter 1.
\begin{itemize}
    \item [\sf (i)] If $(M,d)$ and $(N,\rho)$ are metric spaces and $f\colon M^n\to N$ is uniformly continuous, a modulus of uniform continuity of $f$ is a function 
    $\delta \colon (0,1)\cap{\rat} \to (0,1)\cap{\rat}$ such that whenever $\boldsymbol{x}=(a_1,...,a_n),\boldsymbol{y}=(b_1,...,b_n)\in M^n$ and $\epsilon \in (0,1)\cap{\rat}$, $\sup\{d(a_i,b_i) : 1\leq i \leq n\}<\delta(\varepsilon)$ implies
    $\rho(f(\boldsymbol{x}),f(\boldsymbol{y}))<\varepsilon$.
    \item [\sf (ii)] A language for metric structures is a set $L$ which consists of constants, functions with an associated
    arity and a modulus of uniform continuity, predicates with an associated arity and a modulus of uniform continuity, and a symbol
    $d$ for a metric.
    \item [\sf (iii)] An $L$-metric structure
    $\modelm$ is a metric space $(M,d^M)$ together with interpretations for each symbol in $L$:
    $c^{\modelm}\in{M}$ for each constant $c\in{L}$; 
    $f^{\modelm} \colon M^n\to M$ is an uniformly continuous function for each $n$-ary function symbol $f\in{L}$; $P^{\modelm} \colon M^n\to [0,1]$ is an uniformly continuous function for each $n$-ary predicate symbol $P\in L$.
\end{itemize}
\end{definition}
From now on, we will be mostly interested in languages for metric structures so we will write \emph{language}
instead of \emph{language for metric structures} whenever there is no room for confusion.

Now we proceed to define the continuous analogue of first-order logic in the context of metric structures.
The definition will be recursive and relies on the notion of syntactical objects called \emph{terms}. We will present
the continuous case here following \cite{eagle}; for an exposition of this in the discrete case, see \cite{tentziegler}.
\begin{definition}
Suppose $\{x_n \colon n<\om\}$ is a set of variables and fix a language $L$.
Continuous first-order logic $L$-\emph{terms}
are defined as follows: All variables and constants are $L$-\emph{terms}.
If $t_1,\dots,t_n$ are $L$-\emph{terms} and $f$ is an $n$-ary function symbol, then $f(t_1,\dots,t_n)$
is an $L$-term.
\end{definition}
The interpretation of a term in a metric structure is defined recursively in the natural way. As we discussed in Section 2, a formula $\varphi$
can also be seen as a function such that for a structure $\modelm$ and $a\in{\modelm}$, 
$\varphi(a)=1$ if and only if $\modelm\models{\varphi(a)}$.
For an $L$-sentence $\varphi$ we write 
$\modelm\models{\varphi}$ if and only if
$\varphi(a)=1$ for every $a\in{M}$.
\begin{definition}
Let $L$ be a language for metric structures.
The continuous first-order logic $L$-formulas are defined recursively as follows:
\begin{itemize}
    \item [\sf (i)] Whenever $t_1$ and $t_2$ are $L$-terms, 
    $d(t_1,t_2)$ is an $L$-formula.
    \item [\sf (ii)] If $t_1,\dots,t_n$ are $L$-terms and 
    $P$ an $n$-ary predicate symbol, then $P(t_1, \dots, t_n)$ is an $L$-formula.
    \item [\sf (iii)] If $\varphi_1, \dots, \varphi_n$ are
    $L$-formulas and $g \colon [0,1]^n\to [0,1]$
    is continuous, then $g(\varphi_1, \dots, \varphi_n)$ is an $L$-formula.
    \item [\sf (iv)] If $\varphi$ is an $L$-formula and $x$ is a variable, then $\text{sup}_x\varphi$ and 
    $\text{inf}_x\varphi$ are $L$-formulas.
\end{itemize} 
\end{definition}
\begin{remark}
{\rm
\leavevmode
\begin{itemize}
    \item [\sf (i)] $\text{sup}_x\varphi$ and 
    $\text{inf}_x\varphi$ can be thought 
    as the quantifiers analogous to $\forall$
    and $\exists$. However, they are
    not quite the same: $\text{inf}_x\varphi=1$
    means that for every $n<\om$ there is an $x$ such that $\varphi(x)>1-\frac{1}{n}$, which does not imply that there is an $x$ such that $\varphi(x)=1$.
    \item [\sf (ii)] If $\modelm$ is an $L$-structure we have $\modelm\models\text{min}\{\varphi,\psi\}$
    if and only if $\modelm\models{\varphi}$
    and $\modelm\models\psi$. Analogously, 
    $\modelm\models\text{max}\{\varphi,\psi\}$ if and only if $\modelm\models\varphi$ or $\modelm\models\psi$. We will write $\varphi\wedge\psi$ and $\varphi\vee\psi$ freely. In addition, continuous functions are also introduced as connectives. Notice that only finitary formulas are allowed.
    \item [\sf (iii)] We can express that the equality $a_1=a_2$ holds in a given $\modelm$
    by saying $\modelm\models ``1-d(a_1,a_2)"$, as this means $d^{\modelm}(a_1^{\modelm},a_2^{\modelm})=0$ where $d^{\modelm}$ is a metric so $a_1^{\modelm}=a_2^{\modelm}$.
    \item [\sf (iv)] One of the main differences that occurs when considering continuous logics is that negation is usually not available, i.e. given a formula $\varphi$,  there is not necessarily a formula $\psi$ such that $\modelm\models\varphi$ if and only if $\modelm\not\models\psi$. We will see in Section 6 that negation seems to require an infinitary disjunction to be expressed.
    \item [\sf (v)] All formulas from discrete first-order logic can easily be translated into formulas of continuous first-order logic. See [9] for a detailed discussion. In Section $6$, we will see an example of a formula that cannot be written in discrete first-order logic, but allows us to approximate estimates among norms involved in the construction of Tsirelson’s space.
    \item [\sf (vi)] In general, the relation $\in$ from the language of set theory does not belong to languages for metric structures and so it does not appear in formulas in these languages. As a consequence, an object that can be defined in ZFC, using the axiom of choice for instance, might not be definable in a given language for metric structures. The point of Gowers' problem, however, is to ask whether a space like Tsirelson's can be defined using ordinary analytic notions, which we take to mean whether it is definable in a language suitable for talking about such analytic notions, as contrasted with a language that arguably is suitable for discussing all of mathematics, but not necessarily in a pleasing, transparent way.
\end{itemize}}
\end{remark}
We restate the Compactness Theorem for continuous first-order logic and omit the proof as it is analogous to the discrete case.
\begin{theorem}
Let $L$ be a language for metric structures and $T$ a theory in continuous first-order logic. If 
$T$ is finitely satisfiable then $T$ is consistent. 
\end{theorem}
The general definition of a logic is one that varies from author to author and requires a discussion of foundational issues that might arise. We omit such a general definition in favour of an intuitive description which fits the purposes of working with compact continuous 
logics and continuous $\infinitary$ (defined below). A \emph{logic for metric structures} $\logic$ assigns to each language for metric structures $L$ the set $\sentl$ of all the
sentences together with the space of structures
$\strl$ as defined in Section 1, and a function
$\val \colon \strl\times\sentl\to [0,1]$ 
given by $\val(\modelm,\varphi)=\varphi(a)$ where
$a$ is any element of $M$; $\val$ assigns to each pair $(\modelm,\varphi)$ the truth value of $\varphi$ in $\modelm$ according to $\logic$.

In what follows, we will only be interested in languages for metric structures, regardless of the logic. As a consequence, we can regard each case as a natural generalization of discrete first-order logic: if $L$ is a language, an $L$-structure $\modelm$ is discrete if it is based on a discrete metric space $(M,d)$ and for every $L(M)$-sentence $\varphi$ we have either $\modelm\models\varphi$ or $\modelm\models 1-\varphi$. Thus, first-order logic is a special case of what we are considering here. The work of Casazza and Iovino \cite{casazzaiovino} on the undefinability of Tsirelson’s space is centred around the notion of (finitary) compact continuous logics, i.e. logics for metric structures for which $\strl$ is compact.

\section{\sf $\cp$-Theory meets Continuous Model Theory}

Given a topological space $X$, we will denote by $C_p(X)$ the space of continuous real-valued functions from $X$ with the pointwise convergence topology, i.e. the topology induced on $C_p(X)$ as a subspace of the product topology on $[0,1]^{X}$.
For our purposes, we will only be interested in sets of functions from $X$ to $[0,1]$ and their closures. The properties we will deal with are inherited by closed subspaces of 
$C_p(X)$ and so no difficulty will arise if one considers the closed subspace 
$C_p(X,[0,1])$ of continuous functions from
$X$ to $[0,1]$ instead of $C_p(X)$.
Unless otherwise specified, we suppose that we are working with an
arbitrary logic for metric structures in this section.

Some definitions and statements differ from the original ones in \cite{casazzaiovino} as those are aimed at the compact case on which their work is centred. Surprisingly, the discrete version of what is proved in \cite{casazzaiovino} or here for
continuous logic is not much easier to prove; the reason for this is that
$C_p(X,2)$ is not much simpler than 
$C_p(X, [0,1])$.

We first introduce the formulas in two variables which will occupy our attention.
\begin{definition}
\leavevmode
\begin{itemize}
    \item  [\sf (i)] Let $L$ be a language. $L'\supseteq{L}$ is a language for pairs of structures from $L$ if $L'$
    includes two disjoint copies $L_0$ and $L_1$ of $L$ and there is a map 
    $\strl\times\strl\to \strlprime$ which assigns to every pair of $L$-structures
    $(\modelm,\modeln)$ an $L'$-structure $\la{\modelm,\modeln}\ra$, where $\modelm$ is an $L_0$-structure and $\modeln$ is an $L_1$-structure. We say that $L'$ is a language for pairs of structures if it is so for some $L$.
    \item [\sf (ii)] If $L'$ is a language for pairs of structures from $L$ and $X, Y$ are function symbols from $L$, we say that a formula $\varphi(X,Y)$ is a formula for pairs of structures from $L$ if $(\modelm,\modeln)\mapsto \varphi(X^{\modelm},Y^{\modeln})=
    \val(\varphi(X,Y),\la{\modelm,\modeln}\ra)$ is
    separately continuous on $\strl\times\strl$.
    For notational simplicity, we will write
    $\varphi(\modelm,\modeln)$ instead of 
    $\val(\varphi(X,Y),\la{\modelm,\modeln}\ra)$.
\end{itemize}
\end{definition}
The definition of formulas for pairs of structures is useful for comparing norms on the same metric space. Consider for example a set $\C$ of structures which are normed spaces $(c_{00},\nnnorm)$ based on $c_{00}$, i.e. the structures $(c_{00},\nnnorm_{l_1}, \nnnorm, e_0, e_1, ...)$ where $\nnnorm$ is a function symbol in $L$ for an arbitrary norm and $\{ e_n : n<\omega\}$ is the standard vector basis of $c_{00}$.
Let $L'$ be a language including two disjoint copies of $L$; then two different $L$-structures 
$(c_{00}, \nnnorm_1)$ and $(c_{00}, \nnnorm_2)$ can be coded as the single $L'$-structure 
$(c_{00}, \nnnorm_1, \nnnorm_2)$.
In this section, ``language" will mean ``language for pairs of structures". For the purposes of \cite{casazzaiovino},
$X$ and $Y$ will always represent function symbols for norms.
\begin{definition}
Let $L$ be a language and $\varphi$ a formula
for pairs of structures
\begin{itemize}
    \item [\sf (i)]The left $\varphi$-type of $\modelm$ is the function $\ltpfim\colon \strl\to [0, 1]$
    given by $\ltpfim(\modeln)=
    \varphi(\modelm,\modeln)$. The space of left 
    $\varphi$-types $\slfi$ is the closure of 
    $\{\ltpfim \colon \modelm\in{\strl}\}$ in $C_p(\strl)$. If $\C$ is a subset of $\strl$ then $\slfi[\C]$ is the closure of
    $\{\ltpfim \colon \modelm\in \C\}$ in $C_p(\C)$, called ``the space of left $\varphi$-types over $\C$".
    \item [\sf (ii)]The right $\varphi$-type of $\modeln$ is the function $\rtpfin\colon \strl\to [0, 1]$ given by $\rtpfin(\modeln)=
    \varphi(\modelm,\modeln)$. The spaces $\srfi$ 
    and $\srfi[\C]$ are defined analogously.
\end{itemize}
\end{definition}
\begin{remark}
{\rm
As in the discrete case, one can recover the left
$\varphi$-type in the classical sense by considering $(\ltpfim)^{-1}\{1\}$.}
\end{remark}
\begin{proposition}
If $\logic$ is a compact logic then $\slfi$ and $\srfi$ are compact.
\end{proposition}
\begin{proof}
If $f$ is a limit point of $\slfi$, then there is a cardinal $\kappa$, an ultrafilter $\U$ over $\kappa$ and a sequence $\la{\modelm_{\alpha} \colon\alpha<\kappa}\ra$ in $\strl$ such that $\lim_{\alpha\to\U}ltp_{\varphi, \modelm_{\alpha}}=f$. To see this, let $\kappa$ be the cardinality of the neighbourhood filter $\mathcal{F}$ at $f$, take an ultrafilter $\U$ extending $\mathcal{F}$ and pick a point in each member of $\mathcal{F}$. The trace of $\U$ on this $\kappa$-sequence $\U$-converges to $f$. Since $\strl$ is compact and $\varphi$ is separately continuous, there is an $\modelm\in{\strl}$ such that $\lim_{\alpha\to\U}\modelm_{\alpha}=\modelm$ and for any $\modeln\in{\strl}$,
$\lim_{\alpha\to{\U}}
\varphi(\modelm_{\alpha},\modeln)=
\varphi(\modelm,\modeln)=\ltpfim\in\slfi$.
Then $\slfi$ is a closed subset of 
$[0, 1]^{\strl}$.
\end{proof}
\qed
\begin{definition}
Let $\C$ be a subset of $\strl$ and suppose $\{\modelm_{\alpha} \colon \alpha<\kappa\}\subseteq{\C}$ is such that
$t=\ltpfim$, then we say that $\modelm$ is a realization of $t$ in $\C$ and $t$ is a left $\varphi$-type of $\modelm$ over $\C$.
\end{definition}
We have just showed that the classical result stating that, when the logic is compact, for every type there
is a model in which the type is realized, holds in the continuous case. Notice that
$\strl$ can be replaced
by any closed subset and the result still holds with the same proof.

Now we are ready to introduce the essential concept of stability in continuous logics:
\begin{definition}
Let $\C\subseteq{\strl}$. A formula  for pairs of structures $\varphi$ is stable on $\C$ if and only if whenever $\la{\modelm_{i} \colon i<\omega}\ra$ and 
$\la{\modeln_{\j} \colon j<\omega}\ra$
are sequences in $\C$ and $\U$ and $\V$ ultrafilters on $\omega$, respectively, we have
$$\lim_{i\to\U}
\lim_{j\to\V}
\varphi(\modelm_{i},\modeln_{j})=
\lim_{j\to\V}
\lim_{i\to\U}
\varphi(\modelm_{i},\modeln_{j}).$$
\end{definition}
Now we can state a $\cp$-theoretic result of V. Pt\'ak \cite{ptak} (see also S. Todorcevic \cite{todorcevic}) and the respective model-theoretic version which gives an important characterization of stability:
\begin{theorem}[Pt\'ak]
If $X$ is compact then a pointwise bounded
$A\subseteq{C_p(X)}$ has a compact closure
if and only if it satisfies the following double limit condition: whenever
$\la{f_n \colon n<\om}\ra$ and $\la{x_n \colon n<\om}\ra$ are sequences in $A$ and $X$ respectively, the double limits 
$\lim_{n\to\infty}\lim_{m\to\infty} f_n(x_m)$ and 
$\lim_{m\to\infty}\lim_{n\to\infty} f_n(x_m)$ are equal whenever they both exist.
\end{theorem}
\begin{theorem}
Let $L$ be a language for pairs of structures, $\varphi$ an $L$-formula
for pairs of structures and $\C\subseteq{\strl}$ be compact.
Then, the following are equivalent:
\begin{itemize}
    \item [\sf (i)] $\varphi$ is stable on $\C$.
    \item [\sf (ii)] There is a separately continuous function $F \colon \slfi[\C]\times\srfi[\C]\to [0, 1]$
    such that $F(\ltpfim,\rtpfin)=\varphi(\modelm,\modeln)$ for every $\modelm, \modeln\in{\strl}$.
\end{itemize}
\end{theorem}
\begin{proof}
${\sf (ii)}\rightarrow {\sf (i)}$ is immediate.
${\sf (i)}\rightarrow {\sf (ii)}$: 
We have $A=\{\ltpfim \colon \modelm\in{\C}\}\subseteq{C_p(\strl)}$.
Also, $\slfi[\C]=\bar{A}\cap{C_p(\strl)}$.
Clearly, stability implies the
double limit condition of Pták’s theorem and so $\slfi[\C]$ is compact. By symmetry, 
$F(\ltpfim,\rtpfin)=\varphi(\modelm, \modeln)$ is separately continuous. 
\end{proof}
\qed

Now we can introduce the notion of definability:
\begin{definition}
Let $L$ be a language for pairs of structures, $\varphi$ be a formula
for pairs of structures and $\C\subseteq{\strl}$. A function
$\tau \colon \srfi[\C]\to [0, 1]$ is a left
global $\varphi$-type over $\C$ if there is
a sequence $\la{\modelm_{\alpha} \colon\alpha<\kappa}\ra$ in $\C$ such that for every type $t\in\srfi$, say
$t=\lim_{\beta\to \V}rtp_{\varphi,\modeln_{\beta}}$, we have
$\tau(t)=\lim_{\alpha\to \U}\lim_{\beta\to \V}\phi(\modelm_{\alpha},
\modeln_{\beta})$.
We say that $\tau$ is explicitly definable
if it is continuous.
If a left $\varphi$-type $p\in{\slfi[\C]}$
is given by $p=\lim_{\alpha\to\U}ltp_{\varphi,\modelm_{\alpha}}$, we say that $p$ is explicitly definable if $\tau \colon \srfi\to [0, 1]$ given by 
$\tau(t)=
\lim_{\alpha\to \U}\lim_{\beta\to \V}
\varphi(\modelm_{\alpha},\modeln_{\beta})$ yields an explicitly definable left
global $\varphi$-type.
\end{definition}
The previous definition is sound because the continuity of $\tau$ makes it into an allowable formula which
determines, as in the discrete case, which formulas belong to the type $t$, by taking the pre-image of 1 as usual.

Now Pt\'ak’s theorem (or its model-theoretic version) can be restated as the equivalence between stability and definability in compact continuous logics, following 
\cite{casazzaiovino}:

\begin{theorem}
Let $\logic$ be a compact logic, $L$ a language for pairs of structures, $\C\subseteq{\strl}$  compact and $\varphi$
a formula for pairs of structures. Then, the following are equivalent:
\begin{itemize}
    \item [\sf (i)] $\varphi$ is stable on $\C$.
    \item [\sf (ii)] If $\tau$ is a global left (or right) $\varphi$-type over $\C$, then $\tau$ is explicitly definable.
    \item [\sf (iii)] For every $L$-structure $\modelm$, the left and right $\varphi$-types of $\modelm$ are explicitly definable. 
\end{itemize}
\end{theorem}
This model-theoretic machinery that is built upon $\cp$-theoretic results constitutes the grounds on which Casazza and Iovino proved the undefinability of Tsirelson’s space from sets of Banach spaces including $l_p$ or $c_0$ in compact continuous logics. Most of the remainder of the proof of this fact involves analysis rather
than model theory or topology so the interested reader is referred to 
\cite{casazzaiovino}.
Pt\'ak’s theorem relies on a well-known theorem of Grothendieck which states that, when $X$ is countably
compact, the closure of any $A\subseteq C_p(X)$  which is countably compact in $C_p(X)$ (every infinite subset of $A$ has a limit point in $C_p(X)$) is compact. The spaces for which this property holds, i.e. the closure of any subspace which is countably compact in $C_p(X)$ is compact,
will be called \textit{weakly Grothendieck} spaces; when $C_p(X)$ has this property hereditarily, we say that $X$ is a \textit{Grothendieck} space. Grothendieck spaces have turned out to be relevant in work in progress concerning definability, allowing one to prove stronger results, but they are not involved in what we present here and so we work only with weakly Grothendieck spaces. See \cite{arhangelskii2} for a more detailed discussion on variations of Grothendieck spaces.
Notice that $C_p(X, [0, 1])$ is a closed subspace of $C_p(X)$ and so if $X$ is a (weakly) Grothendieck space, any closed
countably compact subset of 
$C_p(X, [0, 1])$ is compact. Similarly,
$C_p(X, 2)$ is a closed subspace of 
$C_p(X, [0,1])$ and so one may deduce similar conclusions.

The main problem to be overcome in order to find an analogue of the previous definability theorem is to
find conditions on a space $X$ so that
$\bar{A}\cap{C_p(X)}$ being compact is equivalent to a double (ultra)limit
condition as in Pt\'ak’s theorem. A natural place to start is to consider weakly Grothendieck spaces in order to test
statements such as the ones in \cite{casazzaiovino} in the context of logics for which the space of types is weakly Grothendieck,
thus generalizing the compact case.
\bigbreak

{\bf \sf Problem 1:} What conditions are to be satisfied by $X$ so that it is a weakly Grothendieck space?
\bigbreak

There is a variety of disparate conditions which imply $X$ is a weakly Grothendieck space: A. V. Arhangel’ski\u\i \;  showed in \cite{arhangelskii2} that all Lindel\"of $\Sigma$-spaces (i.e. continuous images of spaces that can be
perfectly mapped onto separable metrizable spaces) are weakly Grothendieck, and V. Tkachuk \cite{tkachuk} proved that
every $\sigma$-compact space is weakly Grothendieck. We collect some of these results in the following theorem:
\begin{theorem}\label{grot}
A space $X$ is weakly Grothendieck if it satisfies any of the following conditions:
\begin{itemize}
    \item [\sf (i)] $X$ is countably compact.
    \item [\sf (ii)] $X$ is $\sigma$-compact.
    \item [\sf (iii)] $X$ is Lindel\"of $\Sigma$.
\end{itemize}
\end{theorem}
Theorem \ref{grot} {\sf (i)} is Grothendieck’s theorem. 
The proof of Theorem \ref{grot} {\sf (ii)}
can be reduced to {\sf (iii)} as every $\sigma$-compact space is Lindel\"of $\Sigma$ - see \cite{tkachuk}. The proof for the case in which
$X$ is Lindel\"of $\Sigma$ follows from a variation by Arhangel'ski\u \i \ of a famous theorem of Baturov (see \cite{arhangelskii3} and \cite{arhangelskii}).
\begin{remark}
{\rm 
It is a non-trivial fact that if $Y$ is 
a dense subspace of $X$ and $Y$ is Grothendieck, then $X$ is Grothendieck; the proof requires the full strength of being a Grothendieck space;
see, for instance, \cite{arhangelskii2}. Then, for Grothendieck spaces, we can restrict our attention to dense subspaces in Theorem
5.11. More specifically, we have: if having a property $P$ implies that $X$ is Grothendieck, then having a dense
subspace with the property $P$ also implies that $X$ is Grothendieck. So, if $X$ is Hausdorff \emph{$k$-separable}, i.e.~has
a dense $\sigma$-compact subspace, then it is Grothendieck because Lindel\"of $\Sigma$ spaces, in particular $\sigma$-compact spaces, are Grothendieck (see \cite{arhangelskii2}). We thank V. V. Tkachuk for pointing this out to us. In
the next section, we will be able to extend results in \cite{casazzaiovino} for logics for which the spaces of types are
weakly Grothendieck and first countable. As an example, we will work with countable fragments of a continuous
infinitary logic, $\infinitary$,
for which the spaces of types are metrizable and separable, in which case they have a $\sigma$-compact dense subspace and thus are (weakly) Grothendieck.
}
\end{remark}
The separable case follows from our previous discussion: countable spaces are Grothendieck since they are $\sigma$-compact. By having a Grothendieck dense subspace, separable spaces are Grothendieck.

It can be seen in Arhangel’ski\u\i \; \cite{arhangelskii} that the conditions in theorem 5.11 can yield stronger results, e.g. if $X$ is countably compact, then being pointwise bounded is enough for a subset of $C_p(X)$ 
to have compact closure. Moreover, if $A$ 
is countably compact in $C_p(X)$ for countably compact $X$, then, if $\bar{A}$
is the closure of $A$ in $\real^{X}$, then the closure $\bar{A}\cap{C_p(X)}$ is an Eberlein compactum (i.e. metrizable and homeomorphic to a compact subset of $C_p(Y)$
for some compact metrizable $Y$). This motivates the following problem:
\bigbreak
{\bf \sf Problem 2:} Find weaker conditions under which the closure of subspaces that are countably compact in $C_p(X)$ are just compact.
\bigbreak
Answers to the previous questions would light the road towards extending the results of Casazza and
Iovino to further logics. In \cite{casazzaiovino}, most  results assume that the logic is
countably compact as this allows them to work only with sequences and ultrafilters over $\om$. This is useful
because it allows one to use Theorem 3.8 to test stability (and definability). In the next section, we
generalize some of these results for logics for which the space of types is a first countable weakly Grothendieck space.

\section{\sf Infinitary Continuous Logics and Final Remarks}

In this section we present a double (ultra) limits condition which is equivalent to 
$\bar{A}\cap{C_p(X)}$ being compact when $X$ is first countable and weakly Grothendieck.
As an application, we extend some results from \cite{casazzaiovino} to non-compact
continuous logics with more expressive power for doing analysis.

Recall that a subspace $Y$ of an space $X$
is relatively compact if $\bar{Y}$ is compact in $X$.
\begin{theorem}\label{6.1}
Suppose $X$ is a first countable weakly Grothendieck space. If $A$ is a subset of $C_p(X, [0, 1])$, then $A$ is relatively compact in $C_p(X)$ if and only if it satisfies the following double limit condition: for every pair of sequences
$\la{f_n \colon n<\om}\ra$ and $\la{x_m \colon m<\om}\ra$ and ultrafilters $\U$ and $V$ in $\beta\om$
$$\lim_{n\to \U}\lim_{m\to \V} f_n(x_m)=
\lim_{m\to \V}\lim_{n\to \U} f_n(x_m)$$
whenever the limit $\lim_{m\to \V}x_m$ exists.
\end{theorem}
\begin{proof}
Suppose $\bar{A}\cap{C_p(X)}$ is not compact in $C_p(X)$. Then $\bar{A}\cap{C_p(X)}$ is not countably compact, so let $\la{f_n \colon n<\om}\ra$ be a sequence in $A$ with no limit points in $\bar{A}\cap{C_p(X)}$.
Notice that $\bar{A}$ is compact in $[0, 1]^X$ since it is a closed subset, so each ultralimit of the sequence exists, is a
limit point and is discontinuous.
Take a non-principal ultrafilter $\U$ over $\omega$ and let $\lim_{n\to \U}f_n=g$.
Notice that $g$ is discontinuous by assumption. Then there is an $\varepsilon>0$ and a $y\in{X}$ such that for every neighborhood $U$ of $y$, there is a $y'\in{U}$ such that $|g(y)-g(y')|>\varepsilon$. 
Let $\B$ be a local base at $y$ and consider $\B'=\{\cap{\s} \colon \s\in{[\B]^{<\om}}\}$.
Clearly $\B'$ is also a local base at $y$.
Take an enumeration $\B'=\{B_m \colon m<\om\}$ and, for each $m<\om$, choose $x_m\in{B_m}$ such that $|g(x_m)-g(y)|>\varepsilon$.
Now, for every $k<\om$ let
$A_k=\{m<\om \colon x_m\in{B_k}\}$; 
notice that $\{A_k \colon k<\om\}$ is centred so it can be extended to an ultrafilter $\V\in\beta\om$.
Let $U$ be a neighbourhood of $y$ and let $k<\om$ be such that $B_k\subseteq{U}$.
Then $A_k\subseteq\{m<\om \colon x_m\in{U}\}$. Thus, $\lim_{m\to \V}x_m=y$
and 
$$\lim_{n\to\U}\lim_{m\to \V}f_n(x_m)=g(y)$$
since each $f_n$ is continuous.
On the other hand, 
$\lim_{m\to \V}\lim_{n\to \U} f_n(x_{m})=
\lim_{m\to\V}g(x_{m})$ 
exists by the compactness of [0, 1].
However by our construction, 
$|g(x_{m})-g(y)|>\varepsilon$ for every 
$m<\omega$.
Then the ultralimits $\lim_{n\to\U}\lim_{m\to\V} f_n(x_{\alpha})$ and 
$\lim_{m\to\V}\lim_{n\to\U} f_n(x_{m})$ are different, a contradiction.
Conversely, suppose $\lim_{m\to\V}x_m$ exists and 
$\bar{A}\ \cap\ {C_p(X)}$ is compact. Then for any sequence $\la{f_n \colon n<\om}\ra$ in $A$ and ultrafilter $\U$ over $\om$, there is a continuous $g=\lim_{n\to\U}f_n$. 
Thus we have that 
$\lim_{n\to\U}\lim_{m\to\V}f_n(x_m)=\lim_{n\to\U}f_n(y)=g(y)$ and also that $\lim_{m\to\V}\lim_{n\to\U}f_n(x_m)=\lim_{m\to\V}f_m(y)=g(y)$.
\end{proof}
\qed

``First countability" is not best, but is convenient to use. In work in progress, we aim for the optimal hypotheses for Theorem \ref{6.1}. Now we can restate Theorem 
\ref{6.1} as a model-theoretic result:
\begin{theorem}
Let $\logic$ be a logic, $L$ a language, and suppose 
$\C\subseteq{\strl}$ is such that the space of $\varphi$-types on $\C$ is first countable and weakly Grothendieck. Then, the following are equivalent:
\begin{itemize}
    \item [\sf (i)] Whenever a left type over $\C$ is given by 
    $t=\lim_{i\to\U}ltp_{\varphi,\modelm_i}$
    and $\la{\modeln_j \colon j<\om}\ra$ is a sequence in $\C$ and $\V\in{\beta\om}$
    we have 
    $$\lim_{i\to\U}\lim_{j\to\V}\varphi(\modelm_i,\modeln_j)=\lim_{j\to\V}\lim_{i\to\U}\varphi(\modelm_i,\modeln_j).$$
    \item [\sf (ii)] Whenever a left type is given by 
    $t=\lim_{i\to\U}ltp_{\varphi,\modelm_i}$
    and $\la{\modeln_j \colon j<\om}\ra$ is a sequence in $\C$
    we have 
    $$\text{sup}_{i<j}\varphi(\modelm_i,\modeln_j)=
    \text{inf}_{j<i}\varphi(\modelm_i,\modeln_j).$$
    \item [\sf (iii)] If $\tau$ is a left $\varphi$-type over $\C$ then $\tau$ is explicitly definable.
\end{itemize}
\end{theorem}
\begin{proof}
That {\sf (i)} and {\sf (ii)} are equivalent follows from Lemma \ref{3.8}. Let $\la{\modeln_{\alpha} \colon \alpha<\kappa}\ra$ be a $\kappa$-sequence in $\C$. If $\tau(\lim_{i\to\U}ltp_{\varphi,\modelm_i})
=\lim_{i\to\U}\lim_{\alpha \to\V}\varphi(\modelm_i,\modeln_\alpha)$ defines a global $\varphi$-type
where $\V$ is an ultrafilter over $\kappa$ define $f_{\alpha}\colon \slfi[\C]\to [0, 1]$ by
$f_{\alpha}(\lim_{i\to\U}ltp_{\varphi,\modelm_i})=\lim_{i\to\U}\varphi(\modelm_i,\modeln_{\alpha})$. Then $A=\{f_{\alpha} \colon \alpha<\kappa\}$
 is a subset of ${C_p(\slfi[\C], [0, 1])}$ and satisfies the double ultralimit condition from the previous theorem. Then $\bar{A}\cap{C_p(\slfi[\C], [0,1])}$ is compact
and $\tau=\lim_{\alpha \to\V}f_j$ is continuous.
That {\sf (iii)} implies {\sf (i)} is immediate.
\end{proof}
\qed

Let $\C$ be the set of all structures which are normed spaces based on $c_{00}$. We introduce the continuous logic formula for pairs of structures used in \cite{casazzaiovino}: for norms
$\nnnorm_1$ and $\nnnorm_2$, let
$$
D(\nnnorm_1,\nnnorm_2)=
\text{sup}\left\{\dfrac{\nnorm_1}{\nnorm_2}
\colon \nnorm_{l_1}=1\right\}. 
$$
Then let
$$
\varphi(\nnnorm_1, \nnnorm_2)=\dfrac{\log{D(\nnnorm_1,\nnnorm_2)}}{1+\log{D(\nnnorm_1,\nnnorm_2)}}.
$$
It is not difficult to see that if
$t=\lim_{i\to\U}ltp_{\phi, \nnnorm_i}$ is realized by a structure $(\modelm, \nnnorm_{*})$, then

\begin{equation}\label{1}
\lim_{i\to\U}
\underset{\nnorm_{l_1}=1}{\text{sup}}
\dfrac{\nnorm_i}{\nnorm}=
\underset{\nnorm_{l_1}=1}{\text{sup}}
\dfrac{\nnorm_{*}}{\nnorm}.
\end{equation}
for any structure $(c_{00}, \nnorm)$ in $\C$.
In particular, taking $\nnnorm_{l_1}=\nnnorm$ yields, for each $x\in{c_{00}}$, 
$\lim_{i\to \U}\nnorm_i=\nnorm_{*}$.
If in addition 
$\nnnorm_1\leq\nnnorm_2\dots
\leq{\nnorm_n}\dots$, then 
$\lim_{i\to \infty}\nnorm_i=\nnorm_{*}$.
Conversely, if
$\nnnorm_1\leq\nnnorm_2\dots
\leq{\nnorm_n}\dots$ and also
$\lim_{i\to \infty}\nnorm_i=\nnorm_{*}$for every $x\in{c_{00}}$, then $(1)$ holds for any $(c_{00}, \nnnorm)$ in $\C$ and any
nonprincipal ultrafilter $\U$ over $\om$.

This motivates the following definition from \cite{casazzaiovino}:
\begin{definition}
Let $\C$ be the set of all structures which are normed spaces based on $c_{00}$, $\varphi$ a formula for a pair of structures
and $\nnnorm_{*}$ be a norm on $c_{00}$.
\begin{itemize}
    \item [\sf (i)] If $\{\nnnorm_i \colon i<\om\}$ is a family of norms on $c_{00}$ we say that $\{ltp_{\varphi,\nnnorm_i} \colon i<\om\}$ determines $\nnorm_{*}$ uniquely if for every $\U\in{\beta\om}$,
    $t=\lim_{i\to\U}ltp_{\varphi,\nnnorm_i}$ is realized and $\nnnorm_{*}$ is its unique realization.
    \item [\sf (ii)] We say that $\nnnorm_{*}$ is uniquely determined by its $\varphi$-type over $\C$ if there is a family of norms 
    $\{\nnnorm_i \colon i<\om\}$ on $c_{00}$ such that $\{ltp_{\varphi,\nnnorm_i} \colon i<\om\}$ determines $\nnnorm_{*}$
    uniquely.
\end{itemize}
\end{definition}
Notice that if $\nnnorm_i$ denotes the 
$i$-th iterate of the Tsirelson norm and 
$\nnnorm_T$ is the Tsirelson norm, then 
$\lim_{i\to \infty}\nnorm_i=\nnorm_T$ for
each $x\in{c_{00}}$ and so we have the following result from \cite{casazzaiovino}:
\begin{proposition}
Let $\logic$ be a logic for metric structures, $L$ a language for pairs of structures, $\C$ the class of structures
$(c_{00}, \nnnorm_{l_1},\nnnorm )$ such that the norm completion of $(c_{00}, \nnnorm)$ 
is a Banach space
including $l_p$ or $c_{00}$ and let $\varphi(X,Y)$ be the formula defined above. Suppose $\{(c_{00},\nnnorm_{l_1},\nnnorm_i) \colon i<\om \}$ is a family of structures in $\C$ 
such that 
$\nnnorm_1\leq\nnnorm_2\dots
\leq{\nnorm_n}\dots$ and the $\varphi$-type
$t=\lim_{i\to \U}ltp_{\varphi,\nnnorm_i}$
is realized by $(c_{00}, \nnnorm_{l_1},\nnnorm_{*})$ in $\strl$, 
then $\{ltp_{\varphi, \nnnorm_i} \colon i<\om\}$ determines uniquely $\nnnorm_{*}$
over $\C$. In
particular, the Tsirelson norm is uniquely determined by its $\varphi$-type over $\C$.
\end{proposition}
The following result is proved in \cite{casazzaiovino} and is purely analytical:
\begin{proposition}
Let $\nnnorm_i$ be the $i$-th iterate 
in the construction of the Tsirelson norm. Then the following
hold:
\begin{itemize}
    \item [\sf (i)] $\underset{\nnorm_{l_1}=1}
    {\text{sup}}\dfrac{\nnorm_i}{\nnorm_j}\leq{1}$ for $i<j$.
    \item [\sf (ii)] $\underset{\nnorm_{l_1}=1}
    {\text{sup}}\dfrac{\nnorm_i}{\nnorm_j}\geq{j}$ for $i>j$.
\end{itemize}
Thus, 
$\underset{i<j}{\text{sup}}\; 
\varphi(\nnnorm_i,\nnnorm_j)
\neq 
\underset{j<i}{\text{inf}}\;
\varphi(\nnnorm_i,\nnnorm_j)$.
\end{proposition}
\begin{theorem}
Let $\logic$ be a logic for which the space of types is first countable and weakly Grothendieck, let $L$ be a language for pairs of structures, let $\C$ be the class of structures $(c_{00}, \nnnorm_{l_1},\nnnorm)$ 
such that the norm
completion of $(c_{00}, \nnnorm)$ is a Banach space including $l_p$ or $c_{0}$,
and let $\nnnorm_T$ be the Tsirelson norm.
Then there
is a formula for pairs of structures $\phi$ such that $\nnnorm_T$ is uniquely determined by its $\varphi$-type over $\C$ and that $\varphi$-type is not explicitly definable over $\C$.
\end{theorem}
\begin{remark}
{\rm
The previous result states that, although Tsirelson's space $T$ is constructed via a limiting process involving only explicitly definable spaces based on $c_{00}$, the completions of which include $l_p$ or $c_0$, $T$ itself is not explicitly definable. A much more general result is proved in \cite{casazzaiovino} for countably compact logics, namely:
if a space is explicitly definable from a class of spaces based on $c_{00}$ which include $l_p$ or $c_0$, then it must also
include $l_p$ or $c_0$.}
\end{remark}

Now we will introduce a continuous infinitary logic which extends the compact case, and we will use our $\cp$-theoretic results to prove the undefinability of Tsirelson’s space in this logic.

In discrete model theory, there are natural generalizations of first-order logic denoted by $\lkom$ where $\kappa$ is an infinite cardinal. $\lkom$ is a logic satisfying:
\begin{itemize}
    \item [\sf (i)] All first-order formulas are formulas in $\lkom$.
    \item [\sf (ii)] If $\lambda<\kappa$ and $\chi=\{\varphi_{\alpha} \colon \alpha<\lambda\}$ is a subset of formulas in $\lkom$, then $\wedge \chi$ and $\vee \chi$ are also formulas in $\lkom$.
    \item [\sf (iii)] If $\varphi$ is a formula in $\lkom$, then $(\exists x)\varphi(x)$ and $(\forall x)\varphi(x)$ are formulas in $\lkom$.
\end{itemize}
One of the main differences between $\lkom$ for $\kappa>\om$ and first-order logic ($\lomom$) is that the compactness
theorem does not hold. Moreover, if a weak version of the Compactness Theorem holds for $\lkom$, namely if
whenever $T$ is a theory in $\lkom$ such that if every subsets of $T$ of size $<\kappa$ is satisfiable then $T$ is satisfiable, then $\kappa$ is a weakly compact cardinal and so cannot be proved to exist in $\zfc$. See, for instance, \cite{jech} and
\cite{kanamori}. Thus, in general, one has to do without compactness. The omitting types theorem (an analogue of
the Baire category theorem), is a useful substitute. For an in-depth study of this result in abstract logics,
see \cite{eagletall}.

Here, we will present $\infinitary$, its continuous version, and some of their topological aspects following Eagle
\cite{eagle}. Many concepts useful for analysis can be defined in continuous $\infinitary$, 
but cannot be defined in the
usual continuous logic \cite{eagle}.
\begin{definition}
Let $L$ be a language for metric structures.
The formulas of $\infinitary(L)$ or
$\infinitary(L)$-formulas are defined recursively as follows:
\begin{itemize}
    \item [\sf (i)] All first-order $L$-formulas are $\infinitary(L)$-formulas.
    \item [\sf (ii)] If $\varphi_1, \dots, \varphi_n$ are $\infinitary(L)$-formulas and $g \colon [0, 1]^n\to [0, 1]$ is continuous then $g(\varphi_1,\dots, \varphi_n)$ is an $\infinitary(L)$-formula.
    \item [\sf (iii)] If $\{\varphi_n \colon n<\om \}$ is a family of $\infinitary(S)$-formulas, then $\text{inf}_n\varphi_n$ and $\text{sup}_n\varphi_n$ are $\infinitary(L)$-formulas. These can be also denoted as $\wedge_n\varphi_n$ and $\vee_n \varphi_n$ respectively.
    \item [\sf (iv)] If $\varphi$ is an $\infinitary(L)$-formula and $x$ is a variable, then $\text{inf}_x\varphi$ and $\text{sup}_x\varphi$ are $\infinitary(L)$-formulas.
\end{itemize}
\end{definition}
An interesting feature of continuous $\infinitary$ is that negation ($\neg$) 
becomes available in the classical sense
\cite{eagle}: if $L$ is a language and $\varphi$ 
is an $\infinitary(L)$-formula, then we can define: $\psi(x)=\vee_n\{\varphi(x)+\frac{1}{n}, 1\}$. Then, $\modelm\models\psi(a)$ 
if and only if there is an $n<\om$ such that $\modelm\models\text{max}\{\varphi(a)+\frac{1}{n},1\}$; this is $\text{max}\{\varphi(a)+\frac{1}{n},1\}=1$,
which is the same as $\varphi(a)\leq{1-\frac{1}{n}}$, 
i.e. $\modelm\not\models\varphi(a)$.
Then, $\modelm\models\psi(x)$ if and only if 
$\modelm\not\models\varphi(x)$, and so $\psi$ corresponds to $\neg{\varphi}$.

It is sometimes useful to restrict one’s attention to (countable) fragments of $\infinitary(S)$,  which are easier
to work with than the full logic.

\begin{definition}
Let $L$ be a language. A fragment $\mathcal{F}$ of $\infinitary(L)$ is a set of $\infinitary(L)$-formulas satisfying:
\begin{itemize}
    \item [\sf (i)] Every first-order formula is in $\mathcal{F}$.
    \item [\sf (ii)] $\mathcal{F}$ is closed under finitary conjunctions and disjunctions.
    \item [\sf (iii)] $\mathcal{F}$ is closed under 
    $\text{inf}_x$ and $\text{sup}_x$.
    \item [\sf (iv)] $\mathcal{F}$ is closed under subformulas.
    \item [\sf (v)] $\mathcal{F}$ is closed under substituting terms for free variables.
\end{itemize}
\end{definition}
Notice that every subset of $\infinitary(L)$ generates a fragment, and that every finite set of formulas generates
a countable fragment. There are two arguments to support the idea of working with countable fragments
of $\infinitary(L)$: firstly, notice that a given proof involves only finitely many formulas which then generate a
countable fragment of $\infinitary(L)$.
Secondly (as noted independently by C. J. Eagle (personal
communication)), if an object is definable in $\infinitary(L)$, then it is definable in a countable fragment by the
same argument.

As the work of Casazza and Iovino deals with continuous logics which are finitary in nature, it is natural to
ask whether their result on the undefinability of Tsirelson’s space can be proved for continuous $\infinitary$, 
which is arguably a more natural language from the point of view of Banach space theorists.

In discrete model theory, countable fragments of $\infinitary$ have been studied previously, for instance M.
Morley \cite{morley} showed that the space of types of a countable fragment of $\infinitary$ is Polish.

\begin{remark}
{\rm
In the continuous case, we note that the space of types of a countable fragment $\mathcal{F}$ of
continuous $\infinitary(L)$
can be seen as a subspace of $[0,1]^{\mathcal{F}}$, which is metrizable and second countable, and
so it is separable -- hence it is Grothendieck -- and first countable.}
\end{remark}

The following result follows from Theorem 6.7 and Remark 6.11:

\begin{theorem}
Let $L$ be a language for pairs of structures, $\C$ the class of structures
$(c_{00}, \nnnorm_{l_1}, \nnnorm)$ such that the norm completion of $(c_{00}, \nnnorm)$ includes $l_p$ or $c_0$ and let $\nnnorm_{T}$ be the Tsirelson norm.
Then there is an $\infinitary(L)$-formula 
for pairs of structures $\phi$ such that 
$\nnnorm_T$ is uniquely determined by its 
$\varphi$-type over $\C$ and that $\varphi$-type is not explicitly definable 
over $\C$.
\end{theorem}
\begin{remark}
Note that even for discrete $\infinitary$ this result is new.
\end{remark}
In ongoing research, we extend this non-definability to even wider classes of logics. Applications of $\cp$-theory 
to model theory are not confined to definability questions. Stay tuned!

\section{\sf Acknowledgements}

We thank Jose Iovino, Christopher Eagle and Xavier Caicedo for valuable comments that have helped us to understand \cite{casazzaiovino}, especially the spaces of types. The second author thanks the Mathematics Department of the University of Texas at San Antonio for its gracious hospitality at a workshop in May 2018 where he had the opportunity to interact with these three and with Eduardo Due\~nez. Various subgroups of these four model theorists are doing important work. We thank Vladimir Tkachuk for pointing us in the right direction toward finding the results we needed in $\cp$-theory.

\section{Postscript}
We have not meant to give the impression that extensions of [CI] constitute the only applications of $\cp$-theory to model theory.  In particular, at the suggestion of the referee we shall say a few words about the work of P. Simon and K. Khanaki.

The great utility of stability in model theory led to the investigation of less stringent conditions, in particular what is now known as NIP (the failure of the Independence Property).  The ``bible'' for NIP is Simon's book \cite{simon1}.  

\begin{definition}
A formula $\varphi(x, y)$ has the \emph{independence property in a model $\mathfrak{U}$} if there is an infinite subset $A$ of the universe for which there is a family $\{b_I : I \subseteq A\}$ such that, for each $a \in A$, $\mathfrak{M} \models \varphi(a, b_I) \Leftrightarrow a \in I$. The formula $\varphi(x, y)$ is \emph{NIP} (or \emph{dependent}) if it does not have the independence property.
\end{definition}

As noted in \cite{simon2} and elsewhere, this is equivalent to the condition that for any model M of the theory, the closure in the type space of a subset of size at most $\kappa$ has cardinality at most $2^\kappa$.  The reader will immediately be reminded of the similar condition characterizing (model-theoretic) stability.  Building on work of Rosenthal \cite{rosenthal1}, \cite{rosenthal2}, Bourgain, Fremlin, and Talagrand \cite{bourgainfremlintalagrand} studied the closure in $\mathbb{R}^X$ for $X$ Polish, of subsets $A$ of $\cp(X)$.  As Simon notes, \cite{bourgainfremlintalagrand} proves that either that closure $\bar{A}$ contains non-measurable functions or every element of $\bar{A}$ is a pointwise limit of a sequence of elements of $A$, in which case $|\bar{A}| \le 2^{|A|}$. Simon proves that this dichotomy explicitly corresponds to the dichotomy between theories that do or do not satisfy the independence property. 

The Baire class 1 functions are classically defined as the pointwise limits of sequences of continuous real-valued functions.  Compact subspaces of $B_1(X)$, the collection of Baire class 1 functions for the topological space $X$, are called \textit{Rosenthal compacta} and have an extensive literature.  We will not try to be exhaustive in our references here.  In addition to the already mentioned \cite{bourgainfremlintalagrand} and \cite{simon2}, we call attention to the surveys of Negrepontis \cite{negrepontis} and Debs \cite{debs} and the deep work of Todorcevic \cite{todorcevic2}.  Debs concentrates on \textit{separable} Rosenthal compacta; a noteworthy aspect of \cite{simon2} is that by replacing sequential convergence by filter convergence, Simon mainly eliminates what Aleksandrov \cite{aleksandrov} called the ``parasite of countability''.  We say ``mainly'', because the use of descriptive set theory in \cite{debs} in the separable case cannot be replicated in the general case.  Moreover, it is in the case of countable theories and models that Simon \cite{simon2} establishes the connection between type spaces and Rosenthal compacta: 

\begin{theorem}[Proposition 2.16 of \cite{simon2}] Let $T$ and $M$ be countable and $\varphi(x, y)$ NIP. Then $\operatorname{Inv}_\varphi(M)$ is a Rosenthal compactum.
\end{theorem}

The relevant definitions are: 

\begin{definition}
Given a model $\mathfrak{U}$ and a submodel $\mathfrak{M}$ with universe $M$, the set of automorphisms of $\mathfrak{M}$ fixing $\mathfrak{U}$ is denoted by $\text{Aut}(\mathfrak{U}/\mathfrak{M})$. A $\varphi$-type $p(x)$ over $\mathfrak{U}$ is \emph{$M$-$\varphi$-invariant} if $\sigma p(x) = p(x)$ for every $\sigma \in \text{Aut}(\mathfrak{U}/\mathfrak{M})$. The set of all invariant $M$-$\varphi$-types is denoted by $\text{Inv}_\varphi(M)$.
\end{definition}

The existence of a \emph{monster model} is a model-theoretic technicality which constitutes an important part of the literature and is relevant in this context. In the previous definition, one is usually interested in $\mathfrak{U}$ being the monster model (a class-size model which embeds all set-size models) and $\mathfrak{M}$ being a set-size submodel. For a detailed introduction to the monster model, see \cite{tentziegler} and for more on invariant types see \cite{simon1}.

Strictly speaking, the study of Rosenthal compacta is not a part of $\cp$-theory, since it involves Baire class 1 functions rather than continuous functions.  However it is a natural extension of $\cp$-theory, and in his address at the conference in the proceedings of which this paper will appear, the second author suggested that topologists study $B_1(X)$ with the same vigour they have applied to $\cp(X)$.

An interesting question is where do Rosenthal compacta appear in the spectrum of special compacta studied by $\cp$-theorists.  Are they \textit{Eberlein}, \textit{Gul'ko}, \textit{Corson}, etc.?  So far, the only positive result we have found in the literature is that 

\begin{theorem}\textup{\sf\cite{debs}}
Gul'ko compacta of weight $\le 2^{\aleph_0}$ are Rosenthal.
\end{theorem}

The deep analysis of \cite{todorcevic2} provides many negative results distinguishing Rosenthal compacta from the other well-known compacta studied by $\cp$-theorists.  See \cite{debs}.  Mentioning properties of interest to those who study such compacta, we have for example: 

\begin{theorem}\textup{\sf\cite{todorcevic2}}
In every Rosenthal compactum, the set of $G_\delta$-points includes a dense metrizable subspace.
\end{theorem}

\begin{theorem}\textup{\sf\cite{todorcevic2}}
Rosenthal compacta have $\sigma$-disjoint $\pi$-bases.
\end{theorem}

\begin{corollary}
Rosenthal compacta satisfying the countable chain condition are separable.
\end{corollary}

In a series of mainly unpublished papers, Karim Khanaki has explored connections between model theory and Banach space theory. His latest, \cite{khanaki2}, covers ground familiar to us, including stability, definability, Grothendieck, etc., as well as topics we are just starting to research, such as NIP. However his work is confined to first order logic (continuous or discrete), i.e.~compact logics, and uses \cite{bourgainfremlintalagrand} and classical analysis rather than $\cp$-theory. We expect to be able to generalize many of his results to continuous $L_{\omega_1,\omega}$ and beyond. 



\begin{thebibliography}{00}

\bibitem{aleksandrov} P. S. Aleksandrov. Some results in the theory of topological spaces, obtained within the last twenty-five years.  Russian Mathematical Surveys, 15(2), 23--83, 1960.

\bibitem{arhangelskii3} A. Arhangel'ski\u \i. On the Lindel\"of degree of topological function spaces and on embeddings in $C_p(X)$. Moscow University Mathematics Bulletin, 45(5):43-45, 1990. 


\bibitem{arhangelskii} A. V. Arhangel’ski\u\i. Topological Function Spaces. Mathematics and Its Applications. Kluwer Academic Publishers Group, Dordrecht, 1992.

\bibitem{arhangelskii2} A. V. Arhangel’ski\u\i. On a theorem of Grothendieck in $\cp$-theory. Topology and its Applications, 80, 21-41, 1997.

\bibitem{benyaacoviovino} I. Ben Yaacov and J. Iovino. Model theoretic forcing in analysis. Annals of Pure and Applied Logic, 158(3), 163-174, 2009.

\bibitem{benyaacovusvyatsov} I. Ben Yaacov and A. Usvyatsov. Continuous first order logic and local stability. Transactions of the American Mathematical Society, 362, 5213-5259, 2010.

\bibitem{bourgainfremlintalagrand} J. Bourgain, D. H. Fremlin, M. Talagrand. Pointwise compact sets of Baire-measurable functions. Amer. J. Math., 100, 845-886, 1978.

\bibitem{casazzaiovino} P. Casazza and J. Iovino. On the undefinability of Tsirelson’s space and its descendants. Submitted. ArXiv preprint 1812.02840, 2018.

\bibitem{casazzashura} P. Casazza and T. J. Shura. Tsirelson’s Space. Lecture Notes in Mathematics. Springer-Verlag, Berlin, 1989.

\bibitem{changkeisler} C. C. Chang and H. J. Keisler. Continuous Model Theory. Annals of Mathematics Studies. Princeton University Press, Princeton, NJ, 1966.

\bibitem{debs} G. Debs. Descriptive aspects of Rosenthal compacta, 205-227 in Recent progress in general topology, III. Ed. K. P. Hart, J. van Mill, P. Simon. Atlantic Press, 2014.

\bibitem{eagle} C. J. Eagle. Omitting types for infinitary [0,1]-valued logic. Annals of Pure and Applied Logic, 165, 913-932, 2014.

\bibitem{eagle2} C. J. Eagle. Topological Aspects of Real Valued Logic. University of Toronto, Ph.D. Thesis, 2015.

\bibitem{eagletall} C. J. Eagle and F. D. Tall. Omitting types and the Baire category theorem. ArXiv preprint 1710.05889, 2017.

\bibitem{figieljohnson} T. Figiel and W. Johnson. A uniformly convex Banach space which contains no $l_{p}$. Compositio Mathematica, 29, 179-190, 1974.

\bibitem{gowers} W. T. Gowers. Recent results in the theory of infinite-dimensional Banach spaces. Proceedings of the International Congress of Mathematicians, Vol. 1, 2 (Z\"urich, 1994), 933–942, Birkha\"user, Basel, 1995.

\bibitem{grothendieck} A. Grothendieck. Crit\`eres de compacit\'e dans les espaces fonctionnels g\'en\'eraux. American Journal of Mathematics, 74, 168-186, 1952.

\bibitem{henson} C. W. Henson. Nonstandard hulls of Banach spaces. Israel Journal of Mathematics, 25, 108-144, 1976.

\bibitem{hensoniovino} C. W. Henson and J. Iovino. Ultraproducts in analysis. Analysis and Logic, 262. Journal of the London Mathematical Society Lecture Note Series, 2002.

\bibitem{iovino} J. Iovino. Stable models and reflexive Banach spaces. Journal Symbolic Logic, 64(4),1595–1600, 1999.

\bibitem{iovinoed} J. Iovino (Ed.). Beyond First-Order Model Theory. Monograph and Research Notes in Mathematics. CRC Press, Boca Raton, FL, 2017.

\bibitem{jech} T. Jech. Set Theory. Springer Monographs in Mathematics. Springer-Verlag, Berlin, 2003.

\bibitem{kanamori} A. Kanamori. The Higher Infinite. Springer Monographs in Mathematics. Springer-Verlag, Berlin, 2003.

\bibitem{khanaki} K. Khanaki, Stability, NIP, and NSOP; model theoretic properties of formulas via topological properties of function spaces.  ArXiv arXiv:1410.3339, 2015.

\bibitem{khanaki2} K. Khanaki, Correspondences between model theory and Banach space theory. ArXiv 1512.08691: 1512.08691v2 [math.LO] 5 Dec 2018.

\bibitem{khanakipillay} K. Khanaki and A. Pillay, Remarks on NIP in a model. ArXiv 1706.04674, 2018. 

\bibitem{morley} M. Morley. Applications of topology to $\infinitary$. Proceedings of Symposia in Pure Mathematics, 25, 233-240, 1974.

\bibitem{negrepontis} S. Negrepontis, Banach spaces and topology, 1045-1142 in Handbook of Set-Theoretic Topology, ed. K. Kunen and J. E. Vaughan. Elsevier, 1984.

\bibitem{ptak} V. Pt\'ak. A combinatorial lemma on the existence of convex means and its applications to weak compactness. Convexity, ed. V. L. Klee. Proceedings in Symposia in Pure Mathematics, Volume 7, 437-450, 1963.

\bibitem{rosenthal1} H. P. Rosenthal, A characterization of Banach spaces not containing $\ell_1$. Proceedings of the National Academy of Sciences, 71, 2411-2413, 1974.

\bibitem{rosenthal2} H. P. Rosenthal, Pointwise compact subsets of the first Baire class, American Journal of Mathematics, 99, 362-378, 1977.

\bibitem{shelah} S. Shelah. Stability, the f.c.p, and superstability; model theoretic properties of formulas in first order theory. Annals of Mathematical Logic, 3(3), 271-362, 1971.

\bibitem{shelah2} S. Shelah. Classification Theory and the Number of Nonisomorphic Models. Studies in Logic and the Foundations of Mathematics. North-Holland Publishing Co., Amsterdam, 1990.

\bibitem{simon1} P. Simon, A guide to NIP theories, Lecture Notes in Logic, Cambridge University Press, Cambridge, 2015. 

\bibitem{simon2} P. Simon, Rosenthal compacta and NIP formulas. Fundamenta Mathematicae, 231(1), 81-92, 2015. 

\bibitem{tentziegler} K. Tent and M. Ziegler. A Course in Model Theory. Lecture Notes in Logic, Cambridge University Press, Cambridge, 2012.

\bibitem{tkachuk} V. V. Tkachuk. Lindelof $\Sigma$: An Omnipresent Class. Revista de la Real Academia de Ciencias Exactas Físicas y Naturales, 104 (2), 221-244, 2010.

\bibitem{tkachuk2} V. V. Tkachuk. A $\cp$-Theory Problem Book, v.1-4. Problem Books in Mathematics. Springer, New York, 2011-2015.

\bibitem{todorcevic} S. Todorcevic. Topics in Topology. Lecture Notes in Mathematics. Springer-Verlag, Berlin, 1997.

\bibitem{todorcevic2} S. Todorcevic, Compact subsets of the first Baire class.  Journal of the American Mathematics Society, 12, 1179-1212, 1999.

\bibitem{tsirelson} B. S. Tsirel'son. It is impossible to imbed $l_p$ or $c_0$ into an arbitrary Banach space. Funkcional Anal I Priložen, 8(2), 57-60, 1974.

\bibitem{young} N. J. Young. On Pt\'ak's double-limit theorems. Proceedings of the Edinburgh Mathematical Society (2), 17, 194-200, 1971.



\end{thebibliography}
\end{document}